\documentclass[a4paper,11pt]{amsart}
\usepackage{amssymb}
\usepackage{latexsym}
\usepackage{amsmath}
\usepackage{enumerate}

\usepackage{amsmath, hyperref}
\newtheorem{theorem}{Theorem}[section]
\newtheorem{lemma}[theorem]{Lemma}
\newtheorem{corollary}[theorem]{Corollary}
\newtheorem{proposition}[theorem]{Proposition}

\theoremstyle{definition}
\newtheorem{definition}[theorem]{Definition}
\newtheorem{question}[theorem]{Question}
\newtheorem{fact}[theorem]{Fact}

\numberwithin{equation}{section}

\title{Finding the limit of incompleteness I}
\author{Yong Cheng}
\address{School of Philosophy, Wuhan University, Wuhan 430072 Hubei, Peoples Republic of China}
\email{world-cyr@hotmail.com}
\thanks{I would like to thank Prof.~ Albert Visser for inspiring discussions of the topic in this paper, insightful comments on the original version of the paper  and for introducing some  papers to me. I would also like to thank  Prof.~ Emil Je\v{r}\'{a}bek for his patient explanations of his work to me. I would like to thank Prof.~ Ali Enayat  and Prof.~ V.Yu. Shavrukov for their comments on my paper. I would like to thank the referee for providing detailed and  helpful comments for improvements.  This paper is the research result of The National Social Science Fund of China for general project ``Research on the limit of incompleteness and the intensional problem of incompleteness" (project no: 18BZX131). I would like to thank the fund support by The National Social Science Fund of China for general project.}

\subjclass[2010]{03F40, 03F30, 03D35}

\keywords{G\"{o}del's first incompleteness theorem, Interpretation, Essential undecidability,  Robinson's $\mathbf{R}$}

\begin{document}

\begin{abstract}
In this paper, we  examine the limit of applicability of G\"{o}del's first incompleteness theorem ($\sf G1$ for short). We first define the notion ``$\sf G1$ holds for the theory $T$". This paper is motivated by the following question: can we find a theory with a minimal degree of interpretation for which $\sf G1$  holds. To approach this question, we first examine the following question: is there a theory $T$  such that Robinson's $\mathbf{R}$ interprets $T$ but $T$ does not interpret $\mathbf{R}$ (i.e.~ $T$ is weaker than $\mathbf{R}$ w.r.t.~ interpretation) and $\sf G1$ holds for $T$?
In this paper, we show that there are many such theories based on Je\v{r}\'{a}bek's work using some model theory. We prove that for each recursively inseparable pair $\langle A,B\rangle$,  we can construct a  r.e.~ theory $U_{\langle A,B\rangle}$ such that $U_{\langle A,B\rangle}$ is weaker than $\mathbf{R}$ w.r.t.~ interpretation and  $\sf G1$  holds for $U_{\langle A,B\rangle}$. As a corollary, we answer a question from Albert Visser. Moreover, we prove that for any Turing degree $\mathbf{0}< \mathbf{d}<\mathbf{0}^{\prime}$,  there is a theory $T$ with Turing degree $\mathbf{d}$ such that $\sf G1$  holds for $T$ and  $T$ is weaker  than $\mathbf{R}$ w.r.t.~ Turing reducibility. As a corollary, based on Shoenfield's work using some recursion theory, we show that there is no theory with a minimal degree of Turing reducibility for which $\sf G1$  holds.
\end{abstract}

\maketitle

\section{Introduction}
G\"{o}del's incompleteness theorem is one of the most remarkable results in the foundation of mathematics and has had great influence in logic, philosophy, mathematics, physics and computer science, as discussed in \cite{metamathematics}, \cite{Smith 2007}. G\"{o}del proved his incompleteness theorems in \cite{Godel 1931 original proof} for a certain formal
system $\mathbf{P}$ related to Russell-Whitehead's Principia Mathematica and based on the
simple theory of types over the natural number series and the Dedekind-Peano
axioms (see \cite{Beklemishev 45}, p.3). The following theorem is a modern
reformulation of G\"{o}del's first incompleteness theorem (where $\mathbf{PA}$
refers to the first order theory commonly known as Peano Arithmetic).

\begin{theorem}[G\"{o}del-Rosser, First incompleteness theorem $(\sf G1)$]~
If $T$ is a recursively axiomatized consistent extension of $\mathbf{PA}$, then $T$ is incomplete.
%\begin{enumerate}[(1)]
  %\item
  %if $\mathbf{PA}$ is consistent, then there is a true sentence about arithmetic which is not provable in $\mathbf{PA}$.
  %\item G\"{o}del's second incompleteness theorem $(\sf G2)$:  If $T$ is a recursively axiomatized consistent extension of $\mathbf{PA}$, then the consistency of $T$ is not provable in $T$.
%\end{enumerate}
\end{theorem}

The following is a well known open question about $\sf G1$.
\begin{question}\label{big qn}
Exactly how much arithmetical
information from $\mathbf{PA}$ is needed for the proof of $\sf G1$?
\end{question}

The notion of interpretation provides us a method to compare different theories in different languages (for the definition of interpretation, see Section \ref{section 2}).
Given theories $S$ and $T$, \emph{$S\unlhd T$}  denotes that $S$ is interpretable in $T$ (or $T$ interprets $S$) and
\emph{$S\lhd T$} denotes that  $S$ is interpretable in $T$ but $T$ is not interpretable in $S$. We say  that the theory $S$ and $T$ are \emph{mutually interpretable} if $S\unlhd T$ and $T\unlhd S$. In this paper, we equate a set of sentences $\Gamma$ in the language of arithmetic with the set of G\"{o}del's numbers of sentences in $\Gamma$ (see Section 2 for more details about G\"{o}del's number). Given two arithmetic theories $U$ and $V$,   \emph{$U\leq_T V$}  denotes that the theory $U$ as a set of natural numbers is Turing reducible to the theory $V$ as a set of natural numbers, and \emph{$U<_T V$}  denotes that $U\leq_T V$ but $V\nleq_T U$.

%In this paper, I focus on $\sf G1$.
Note that $\sf G1$ can be generalized via interpretability: there exists a weak recursively axiomatizable consistent subtheory $T$ (e.g.~ Robinson Arithmetic $\mathbf{Q}$) of $\mathbf{PA}$ such that for each recursively axiomatizable consistent theory $S$, if  $T$ is interpretable  in $S$, then $S$ is incomplete (see \cite{undecidable}). To generalize this fact, in the following, we propose a general new notion ``$\sf G1$ holds for  the theory $T$". %which proves us a route to find the limit of incompleteness.

\begin{definition}~
\begin{enumerate}[(1)]
  \item Let $T$ be a recursively axiomatizable consistent theory. We say that \emph{$\sf G1$ holds for  $T$}  if for any recursively axiomatizable consistent theory $S$, if $T$ is interpretable in $S$, then  $S$ is incomplete.
  \item We say the theory $S$ has a \emph{minimal degree of interpretation} if there is no theory $T$ such that $T\lhd S$.
  \item We say the theory $S$ has a \emph{minimal degree of Turing reducibility} if there is no theory $V$ such that $V<_T S$.
  \item In this paper, whenever we say that
\emph{the theory $S$ is weaker than the theory $T$ w.r.t.~ interpretation}, this means that $S\lhd T$.
\item In this paper, whenever we say that
\emph{the theory $S$ is weaker than the theory $V$ w.r.t.~ Turing reducibility}, this means that $S<_T V$.
\end{enumerate}
\end{definition}

Toward Question \ref{big qn}, in this project, we want to examine the following question:
\begin{question}\label{key qn}
Can we find a theory $S$ such  that $\sf G1$  holds for $S$ and $S$ has a minimal degree of interpretation?
\end{question}
%This work is motivated by the following question:

%\begin{question}\label{qn}
%Whether there exists a minimal theory $T$ with respect to the degree of interpretability such that $\sf G1$ holds for $T$?
%\end{question}

It is well known that $\sf G1$ holds for Robinson Arithmetic $\mathbf{Q}$ (see \cite{undecidable}). From \cite{undecidable}, $\sf G1$ also holds for  Robinson's theory $\mathbf{R}$ (for   definitions of $\mathbf{Q}$ and $\mathbf{R}$, we refer to Section \ref{section 2}).
In Section \ref{section 2}, we review some theories mutually interpretable with $\mathbf{Q}$ for which $\sf G1$ holds, and some theories mutually interpretable with $\mathbf{R}$ for which $\sf G1$ holds. As the first step toward Question \ref{key qn}, we propose the following question:

\begin{question}\label{main qn}
Can we find a theory $S$ such that $\sf G1$ holds for $S$ and $S\lhd \mathbf{R}$?
\end{question}

We find that Je\v{r}\'{a}bek essentially answered this question in \cite{Emil}. In this paper, we show that there are many examples of such a theory $S$: for each recursively inseparable pair $\langle A,B\rangle$, we can construct a r.e.~ theory $U_{\langle A,B\rangle}$ such that $\sf G1$ holds for $U_{\langle A,B\rangle}$ and $U_{\langle A,B\rangle}\lhd \mathbf{R}$ based on Je\v{r}\'{a}bek's work using some model theory. As a corollary, we answer a question from Albert Visser. For  Question \ref{key qn}, if we consider the degree of Turing reducibility instead of the degree of interpretation, the answer becomes easier. We show that for any Turing degree $0< \mathbf{d}<0^{\prime}$, there is a theory $S$ such that $\sf G1$  holds for $S$, $S<_T \mathbf{R}$  and  $S$ has Turing degree $\mathbf{d}$ based on Shoenfield's work using some recursion theory.  As a corollary, there is no theory with a minimal degree of Turing reducibility for which $\sf G1$  holds. %The main results of this paper are Theorem \ref{main thm}, Theorem \ref{key theorem} and Corollary \ref{key corollary}.

%: for any recursively inseparable pair, there is a theory $S$ such that $\sf G1$  holds for $S$  and $S\lhd \mathbf{R}$ (see Theorem \ref{main thm}).
%As a corollary, I negatively answer the following question from Albert Visser: would $S$ with $S\unlhd\mathbf{R}$ such that $\sf G1$  holds for  $S$ shares the universality property
%of $\mathbf{R}$ that every locally finitely satisfiable theory is interpretable in it.

%\begin{theorem}\label{main thm}
%For any recursively inseparable pair $\langle A,B\rangle$, there is a theory $U_{\langle A,B\rangle}$ such that $\sf G1$  holds for $U_{\langle A,B\rangle}$  and $U_{\langle A,B\rangle}\lhd\mathbf{R}$.
%\end{theorem}
 %It is proved in \cite{Variants of Robinson} that $\mathbf{R}^{\ast}$ is essentially undecidable. In this paper, we want to prove more  metamathematical results for $\mathbf{R}^{\ast}$.

%There are three kinds of methods for proofs of $\sf G1$: constructive proof (constructing the independent sentence); diagonalization proof (using fixed point theorem); and recursive theoretical proof (using some recursion theory).

The structure of  this paper is as follows. In Section 1, we introduce our research questions and  main results of this paper. In Section 2, we list some basic notions and facts we use in this paper and give a review of theories weaker  than $\mathbf{PA}$ w.r.t.~ interpretation in the literature for which $\sf G1$ holds. In Section 3, we examine the limit of applicability of $\sf G1$ w.r.t.~ interpretation. We prove Theorem \ref{main thm} and answer a question from Albert Visser. In Section 4, we examine the limit of applicability of $\sf G1$ w.r.t.~ Turing reducibility and prove Theorem \ref{key theorem} and Corollary \ref{key corollary}.

%sub-theories of $\mathbf{PA}$
%summary of some important tools and results on meta-mathematics related to incompleteness and undecidability.

%construct many theories weaker than $\mathbf{R}$ w.r.t.~ interpretation for which $\sf G1$ holds.  As a corollary, we answer a question from Albert Visser. Then, based on Shoenfield's work using some recursion theory, we show that there is no theory with a minimal degree of Turing reducibility for which $\sf G1$  holds.
%for any Turing degree $0< \mathbf{d}<0^{\prime}$, there is a theory weaker than $\mathbf{R}$ w.r.t.~ Turing reducibility with Turing degree $\mathbf{d}$ for which $\sf G1$  holds.  Thus,

\section{Preliminaries}\label{section 2}

In this section, we review some basic notions and facts used in this paper. Our notations are standard. %For the theory of recursive function, refer to \cite{Enderton 2011}, \cite{metamathematics} and \cite{Rogers67}.
For books on G\"{o}del's incompleteness theorem, we refer to \cite{Enderton 2001}, \cite{metamathematics}, \cite{Per 97}, \cite{Smith 2007}, \cite{Boolos 93} and \cite{Cheng book 19}. For survey papers on G\"{o}del's incompleteness theorem, we refer to \cite{Beklemishev 45}, \cite{Kotlarski 2004}, \cite{Smorynski 1977}, \cite{Visser 16}, \cite{Cheng 19} and \cite{Cheng 20}. For meta-mathematics of subsystems of $\mathbf{PA}$, we refer to \cite{Metamathematics of First-Order Arithmetic}.

In this paper, a language consists of an arbitrary number of relation and function symbols of arbitrary finite arity.\footnote{We may view nullary functions as constants and nullary relations as propositional
variables.}  For a given theory $T$, we use \emph{$L(T)$} to denote  the language of $T$ and often equate $L(T)$ with the list of non-logical symbols of the language.
For a formula $\phi$ in $L(T)$,  let \emph{$T\vdash\phi$} denote that $\phi$ is provable in $T$ (i.e., there is a finite sequence of formulas $\langle\phi_0, \cdots,\phi_n\rangle$ such that $\phi_n=\phi$, and for any $0\leq i\leq n$, either $\phi_i$ is an axiom of $T$ or $\phi_i$ follows from some $\phi_j\, (j<i)$ by using one inference rule).
Theory $T$ is \emph{consistent} if no contradiction is provable in $T$.   A formula $\phi$ is \emph{independent} of $T$ if $T\nvdash \phi$ and $T\nvdash \neg\phi$.
A theory $T$ is \emph{incomplete}  if there is a sentence $\phi$ in $L(T)$ such that $\phi$ is independent of $T$; otherwise, $T$ is \emph{complete} (i.e., for any sentence $\phi$ in $L(T)$, either $T\vdash\phi$ or $T\vdash \neg\phi$).

In this paper, we
understand that each theory $T$ comes with a fixed arithmetization.
Let $T$ be a recursively axiomatizable theory. Under this fixed arithmetization, we could establish the one-to-one correspondence between formulas of $L(T)$ and natural numbers. Under this correspondence, we can translate metamathematical
statements about the formal theory $T$ into statements
about natural numbers. Furthermore, fundamental metamathematical relations can be translated in this way into
certain recursive relations, hence into relations representable in the theory $T$. Consequently, one can speak about a formal
system of arithmetic and about its properties
as a theory in the system itself! This is the essence of G\"{o}del's
idea of arithmetization.
Under arithmetization, any formula or finite sequence of formulas  can be coded by a natural number (this code is called a G\"{o}del number).
In  this paper, we use \emph{$\ulcorner\phi\urcorner$} to denote the G\"{o}del number of $\phi$.
 For details of arithmetization, we refer to \cite{metamathematics}.

Given a set of sentences $\Sigma$, we say $\Sigma$ is \emph{recursive}  if the set of G\"{o}del numbers of sentences in $\Sigma$ is recursive. A theory $T$ is \emph{decidable}   if the set of sentences provable in $T$ is recursive; otherwise it is \emph{undecidable}. A theory $T$ is \emph{recursively axiomatizable}  if it has a recursive set of axioms (i.e.~ the set of G\"{o}del numbers of axioms of $T$ is recursive) and it is \emph{finitely axiomatized}  if it has a finite set of axioms. A theory $T$ is \emph{recursively enumerable}  (r.e.) if it has a recursively enumerable set of axioms.
A theory $T$
is  \emph{essentially undecidable}  if any recursively axiomatizable consistent extension of $T$ in the same language is undecidable.
A theory $T$ is \emph{essentially incomplete}  if  any recursively axiomatizable consistent extension of $T$ in the same language is incomplete. The theory of completeness/incompleteness  is closely related to the theory of decidability/undecidability.
A theory $T$ is \emph{minimal essentially undecidable}  if $T$ is essentially undecidable and if deleting any axiom of $T$, the remaining theory is no longer essentially undecidable. A theory $T$
is \emph{locally finitely satisfiable} if every finitely axiomatized subtheory of $T$ has
a finite model.

A $n$-ary relation $R(x_1, \cdots, x_n)$ on $\mathbb{N}^n$ is \emph{representable}  in $T$ iff there is a formula $\phi(x_1, \cdots, x_n)$ such that $T\vdash \phi(\overline{m_1}, \cdots, \overline{m_n})$ if $R(m_1, \cdots, m_n)$ holds (for  $n\in \mathbb{N}$, we denote by $\overline{n}$  the corresponding numeral for $n$ in $L(\mathbf{PA})$) and $T\vdash \neg\phi(\overline{m_1}, \cdots, \overline{m_n})$ if $R(m_1, \cdots, m_n)$ does not hold. %We can generalize representation of relations to representation of disjoint pairs of relations: such a disjoint pair $\langle P^{+}, P^{-}\rangle$ where $P^{+}, P^{-}\subseteq \mathbb{N}^n$, is representable by  a formula $\phi(x_1, \cdots, x_n)$  in $T$ if  $T\vdash\phi(\overline{a_1}, \cdots, \overline{a_{n}})$
%for $\langle a_1, \cdots, a_{n}\rangle \in P^{+}$, and
%$T\vdash\neg\phi(\overline{a_1}, \cdots, \overline{a_{n}})$ for $\langle a_1, \cdots, a_{n}\rangle \in P^{-}$.
%A theory $T$ is said to be $\omega$-consistent if there is no formula $\varphi(x)$ such that $T \vdash\exists x \varphi(x)$ and for any $n\in\mathbb{N}$, $T \vdash\neg\varphi(\overline{n})$; and T is $1$-consistent if there is no such a $\Delta^0_1$ formula $\varphi(x)$.
We say that a partial function $f(x_1, \cdots, x_n)$ on $\mathbb{N}^n$ is representable in $T$ iff there is a formula $\varphi(x_1, \cdots, x_n,y)$ such  that $T\vdash \forall y(\varphi(\overline{a_1}, \cdots, \overline{a_{n}},y)\leftrightarrow y=\overline{m})$ whenever $a_1, \cdots, a_n, m\in \mathbb{N}$ are such that $f(a_1, \cdots, a_n)=m$.

 %$I\Sigma_{n}$ is a fragment of $\mathbf{PA}$ obtained by restricting the axiom scheme of induction to $\Sigma_{n}$ formulas.

For the definitions of  translation and interpretation, we follow \cite{Emil}.
Let $T$ be a theory in a language $L(T)$, and $S$ a theory in a language $L(S)$. In its most simple form,
a \emph{translation}  $I$ of language $L(T)$ into language $L(S)$ is specified by:
 \begin{enumerate}[(1)]
  \item an $L(S)$-formula $\delta_I(x)$ denoting the domain of $I$;
  \item for each relation symbol $R$ of $L(T)$, an $L(S)$-formula $R_I$ of the same arity;
  \item for each function symbol $F$ of $L(T)$ of arity $k$, an $L(S)$-formula $F_I$ of arity $k + 1$.
\end{enumerate}

If $\phi$ is an $L(T)$-formula, its $I$-translation $\phi^I$ is an $L(S)$-formula constructed as follows: we reformulate the
formula in an equivalent way so that function symbols only occur in atomic subformulas of the
form $F(\overline{x}) = y$ where $x_i, y$ are variables; then we replace each such atomic formula with $F_I(\overline{x}, y)$,
we replace each atomic formula of the form $R(\overline{x})$ with $R_I(\overline{x})$, and we restrict all quantifiers and
free variables to objects satisfying $\delta_I$. Moreover, we rename bound variables to avoid
variable clashes during the process (see \cite{Emil}).

A translation $I$ of $L(T)$ into $L(S)$ is an \emph{interpretation}  of $T$ in $S$ if $S$ proves:
 \begin{enumerate}[(1)]
  \item for each function symbol $F$ of $L(T)$ of arity $k$, the formula expressing that $F_I$ is total on $\delta_I$:
$\forall x_0, \cdots \forall x_{k-1} (\delta_I(x_0) \wedge \cdots \wedge \delta_I(x_{k-1}) \rightarrow \exists y (\delta_I(y) \wedge F_I(x_0, \cdots, x_{k-1}, y)))$;
  \item the $I$-translations of all theorems of $T$, and axioms of equality.
\end{enumerate}
%Note that $S$ proves the $I$-translations of all sentences provable in $T$.
The simplified picture of translations and interpretations above actually describes only \emph{one-dimensional}, \emph{parameter-free}, and \emph{one-piece} translations (see \cite{Emil}). For the precise and technical
definitions of a \emph{multi-dimensional} interpretation, an interpretation \emph{with parameters}   and a \emph{piece-wise} interpretation, we refer to \cite{paper on Q}, \cite{Visser 11} and \cite{Visser 14}   for the details.

A theory $T$ is \emph{interpretable}  in a theory $S$ if there exists an
interpretation  of $T$ in $S$. If $T$ is interpretable in $S$, then all sentences provable (refutable) in $T$ are mapped, by the interpretation function, to sentences provable (refutable) in $S$. Interpretability can be accepted as a measure of strength of different theories. If $S\lhd T$, then $S$ can
be considered weaker than $T$ w.r.t.~ interpretation; if $S$ and $T$ are mutually interpretable, then $T$ and $S$ are equally strong w.r.t.~ interpretation.
The theory $U$ \emph{weakly interprets}  the theory $V$ (or $V$ is weakly interpretable in $U$) if $V$ is interpretable in some consistent extension  of $U$ in the same language (or equivalently, for some interpretation $\tau$, the
theory $U + V^{\tau}$ is consistent).
%We say $I$ is a \emph{faithful interpretation}   of $T$ in $S$ if $I$ is an interpretation of $T$ in $S$ such that for any sentence $\phi$ in $L(T)$, $T\vdash \phi$ iff $S\vdash \phi^{I}$.

%Let $S\unlhd T$ denote $S$ is interpretable in $T$; let $S\lhd T$ denote $S$ is interpretable in $T$ but $T$ is not interpretable in $S$. $S$ and $T$ are mutually interpretable if $S\unlhd T$ and $T\unlhd S$.
%Let $Int(S)$ denote the degree of interpretation of theory $S$. $Int(T)<Int(S)$ means that $T$ is interpretable in $S$ but $S$ is not interpretable in $T$. $Int(T)=Int(S)$ means that $T$ and $S$ are mutually interpretable.

A general method for establishing the undecidability of theories is developed in \cite{undecidable}.
The following theorem provides us two methods to prove the essentially undecidability of a theory via interpretation and representability.

\begin{theorem}\label{interpretable theorem}~
\begin{enumerate}[(1)]
  \item (\cite[Theorem 7, p.22]{undecidable})   Let $T_1$ and $T_2$ be two theories such that $T_1$ is consistent and $T_2$ is interpretable in $T_1$. We then have: if $T_2$ is essentially undecidable, then $T_1$ is also essentially undecidable.
  \item (\cite[Corollary 2, p.49]{undecidable})   If $T$ is a consistent theory in which all recursive functions are representable, then $T$ is essentially undecidable.
\end{enumerate}
\end{theorem}

In Section 3, we will show that $\sf G1$  holds for the theory $T$ iff $T$ is essentially undecidable. In the following, we review some theories from the literature which are weaker than $\mathbf{PA}$ w.r.t.~ interpretation and which are essentially undecidable (i.e.~ $\sf G1$  holds for them).

Robinson Arithmetic $\mathbf{Q}$ was introduced in \cite{undecidable} by Tarski, Mostowski and
Robinson  as a base axiomatic theory for investigating incompleteness and undecidability.

\begin{definition}\label{def of Q}
Robinson Arithmetic $\mathbf{Q}$  is  defined in   the language $\{\mathbf{0}, \mathbf{S}, +, \cdot\}$ with the following axioms:
\begin{description}
  \item[$\mathbf{Q}_1$] $\forall x \forall y(\mathbf{S}x=\mathbf{S} y\rightarrow x=y)$;
  \item[$\mathbf{Q}_2$] $\forall x(\mathbf{S} x\neq \mathbf{0})$;
  \item[$\mathbf{Q}_3$] $\forall x(x\neq \mathbf{0}\rightarrow \exists y x=\mathbf{S} y)$;
  \item[$\mathbf{Q}_4$]  $\forall x (x+ \mathbf{0}=x)$;
  \item[$\mathbf{Q}_5$] $\forall x\forall y(x+ \mathbf{S} y=\mathbf{S} (x+y))$;
  \item[$\mathbf{Q}_6$] $\forall x(x\cdot \mathbf{0}=\mathbf{0})$;
  \item[$\mathbf{Q}_7$] $\forall x\forall y(x\cdot \mathbf{S} y=x\cdot y +x)$.
\end{description}
\end{definition}

Robinson Arithmetic $\mathbf{Q}$ is  very weak and inadequate to formalize arithmetic: for instance, $\mathbf{Q}$ does not
 even prove that addition is associative.
Robinson showed that any consistent theory that interprets $\mathbf{Q}$ is undecidable and hence
$\mathbf{Q}$ is essentially undecidable (see \cite{undecidable}). The fact that $\mathbf{Q}$ is  essentially undecidable is very useful and can be used to prove the  essentially undecidability of other theories via Theorem \ref{interpretable theorem}. Since $\mathbf{Q}$ is finitely axiomatized, it follows that any
theory that weakly interprets $\mathbf{Q}$ is undecidable. In fact, $\mathbf{Q}$ is minimal essentially undecidable in the sense that if deleting any axiom of $\mathbf{Q}$, then the remaining theory is not essentially undecidable and has a complete decidable extension (see \cite[Theorem 11, p.62]{undecidable}). Nelson  \cite{Nelson 86} embarked  on a program   of investigating how much mathematics can
 be interpreted in Robinson Arithmetic $\mathbf{Q}$: what can be interpreted in $\mathbf{Q}$ but also what cannot  be  interpreted in $\mathbf{Q}$. In fact, $\mathbf{Q}$ represents a rich degree of interpretability since a lot of stronger theories are interpretable in it as it can be shown (e.g.~ using Solovay's
method of shortening cuts (see \cite{Guaspari 79}), one can show that $\mathbf{Q}$ interprets fairly strong theories like $I\Delta_0 + \Omega_1$ on a definable cut).
The Lindenbaum algebras of all recursively enumerable theories that interpret $\mathbf{Q}$ are recursively isomorphic (see Pour-El and Kripke \cite{Kripke 67}).

The theory $\mathbf{PA}$ consists of axioms $\mathbf{Q}_1$-$\mathbf{Q}_2$, $\mathbf{Q}_4$-$\mathbf{Q}_7$ in Definition \ref{def of Q} and the following axiom scheme of induction: $(\phi(\mathbf{0})\wedge \forall x(\phi(x)\rightarrow \phi(\mathbf{S} x)))\rightarrow \forall x \, \phi(x)$  where $\phi$ is a formula with at least one free variable $x$.

Now we first discuss some prominent fragments of $\mathbf{PA}$ extending $\mathbf{Q}$ in the literature.

We define the arithmetic hierarchy $I\Sigma_n$ and $B\Sigma_{n}$ in the language of   $\mathbf{PA}$. An $L(\mathbf{PA})$-formula is bounded (or $\Delta_0$ formula) if all quantifiers occuring in it are bounded, i.e.~ in the form $(\exists x \leq y)\phi$ and $(\forall x \leq y)\phi$. For the definitions of $\Sigma_{n}, \Pi_{n}$ and $\Delta_{n}$ formulas ($n\geq 1$), we refer to \cite{metamathematics}.
\emph{Collection} for $\Sigma^0_{n+1}$ formulas is the following principle:
\[(\forall x< u)(\exists y) \varphi(x,y)\rightarrow (\exists v)(\forall x< u)(\exists y< v) \varphi(x,y),\]
where
$\varphi(x,y)$ is a $\Sigma^0_{n+1}$ formula possibly containing parameters distinct from $u,v$.

The theory $I\Sigma_n$ is $\mathbf{Q}$ plus induction for $\Sigma_n$ formulas and $B\Sigma_{n+1}$ is $I\Sigma_0$ plus collection for $\Sigma_{n+1}$ formulas. It is well known that the following theories form a strictly increasing hierarchy:
\[I\Sigma_0, B\Sigma_{1}, I\Sigma_1, B\Sigma_{2},\cdots I\Sigma_n, B\Sigma_{n+1}, \cdots, \mathbf{PA}.\]

Define $\omega_1(x)=x^{\mid x\mid}$ and
$\omega_{n+1}(x)=2^{\omega_{n}(\mid x\mid)}$ where $|x|$ is the length of the binary expression of $x$. Note that the graphs of these functions
can be defined in our language with the recursive defining equation
provable (see \cite{Metamathematics of First-Order Arithmetic}).
Let $\Omega_{n}$ denote the statement $\forall x\exists y (\omega_{n}(x) =y)$ which says that $\omega_{n}(x)$ is total.
There is a bounded formula ${\sf Exp(x, y, z)}$ such that $I\Sigma_0$ proves that  ${\sf Exp(x, 0, z)} \leftrightarrow z= 1$ and ${\sf Exp(x, Sy, z)} \leftrightarrow \exists t ({\sf Exp(x, y, t)} \wedge z=t\cdot x)$ (see \cite[Proposition 2, p.299]{Interpretability in Robinson's Q}). However,  $I\Sigma_0$ cannot prove  the totality of ${\sf Exp(x, y, z)}$. Let $\mathbf{exp}$ denote the statement postulating the totality of the exponential function $\forall x\forall y\exists z  {\sf Exp(x, y, z)}$.

\begin{theorem}[\cite{Petr 93}, \cite{Interpretability in Robinson's Q}]~
\begin{enumerate}[(1)]
  \item  For any $n\geq 1$, $I\Sigma_0+\Omega_n$ is interpretable in $\mathbf{Q}$ (see \cite[Theorem 3, p.304]{Interpretability in Robinson's Q}).
  \item $I\Sigma_0+ \mathbf{exp}$ is not interpretable in $\mathbf{Q}$.\footnote{See \cite[Theorem 6, p.313]{Interpretability in Robinson's Q}. Solovay proved that $I\Sigma_0 + \neg \mathbf{exp}$ is interpretable in $\mathbf{Q}$ (see \cite[Theorem 7, p.314]{Interpretability in Robinson's Q}).}
  \item  $I\Sigma_1$ is not interpretable in $I\Sigma_0+ \mathbf{exp}$ (see \cite[Theorem 1.1]{Petr 93}, p.186).
      \item $I\Sigma_{n+1}$ is not interpretable in $B\Sigma_{n+1}$ (see \cite[Theorem 1.2]{Petr 93}, p.186).
      \item  $B\Sigma_{1}+ \mathbf{exp}$ is interpretable in $I\Sigma_{0}+ \mathbf{exp}$ (see \cite[Theorem 2.4]{Petr 93}, p.188).
      \item For each $n\geq 1$, $B\Sigma_{1}+\Omega_{n}$ is interpretable in $I\Sigma_{0}+\Omega_{n}$ (see \cite[Theorem 2.5]{Petr 93}, p.189).
      \item  For each $n\geq 0$, $B\Sigma_{n+1}$ is interpretable in $I\Sigma_{n}$ (see \cite[Theorem 2.6]{Petr 93}, p.189).
\end{enumerate}
\end{theorem}

As a corollary,  we have:
\begin{enumerate}[(1)]
  \item The theories $\mathbf{Q},  I\Sigma_0, I\Sigma_0+\Omega_{1}, \cdots, I\Sigma_0+\Omega_{n}, \cdots, B\Sigma_1, B\Sigma_1+\Omega_{1}, \cdots, B\Sigma_1+\Omega_{n}, \cdots$ are all mutually interpretable;
  \item $I\Sigma_0 + \mathbf{exp}$ and $B\Sigma_1 + \mathbf{exp}$ are  mutually interpretable;
   \item For $n\geq 1$, $I\Sigma_n$ and $B\Sigma_{n+1}$ are  mutually interpretable;
    \item  $\mathbf{Q}\lhd I\Sigma_0 + \mathbf{exp}\lhd I\Sigma_1\lhd I\Sigma_2\lhd\cdots\lhd I\Sigma_n\lhd\cdots\lhd \mathbf{PA}$.
\end{enumerate}

 %\footnote{It follows from this that we have the second incompleteness theorem for all extensions of $\mathbf{Q}$. See \cite{Metamathematics of First-Order
%Arithmetic}.}

Now we discuss some weak theories in the literature which are  mutually interpretable with  $\mathbf{Q}$.

It is interesting to compare $\mathbf{Q}$ with its bigger brother $\mathbf{PA}^{-}$.
The theory $\mathbf{PA}^{-}$ is the theory of commutative, discretely
ordered semi-rings with a minimal element plus the subtraction axiom. The theory $\mathbf{PA}^{-}$ has the following axioms with $L(\mathbf{PA}^{-})=L(\mathbf{PA})\cup\{\leq\}$: (1) $x + 0 = x$; (2) $x + y = y + x$; (3) $(x + y) + z = x + (y + z)$; (4) $x \cdot 1 = x$; (5) $x \cdot y = y \cdot x$; (6)$(x \cdot y) \cdot z = x \cdot (y \cdot z)$; (7) $x \cdot (y + z) = x \cdot y + x \cdot z$; (8)$x \leq y \vee y \leq x$; (9) $(x \leq y \wedge y \leq z)\rightarrow x \leq z$; (10) $x + 1 \nleq x$; (11) $x \leq y \rightarrow (x = y \vee x + 1 \leq y)$; (12) $x \leq y \rightarrow x + z \leq y + z$; (13) $x \leq y \rightarrow x \cdot z \leq y \cdot z$; (14)$x \leq y \rightarrow \exists z (x + z = y)$.
From \cite{paper on Q}, $\mathbf{PA}^{-}$ is interpretable in $\mathbf{Q}$, and hence $\mathbf{PA}^{-}$ is mutually interpretable with $\mathbf{Q}$.

Let $\mathbf{Q}^+$ be the extension of $\mathbf{Q}$ with the following extra axioms ($L(\mathbf{Q}^+)=L(\mathbf{Q})\cup\{\leq\}$):
\begin{description}
  \item[$\mathbf{Q}_8$] $(x + y) + z = x + (y + z)$;
  \item[$\mathbf{Q}_9$] $x \cdot (y + z) = x \cdot y + x \cdot z$;
  \item[$\mathbf{Q}_{10}$] $(x \cdot y) \cdot z = x\cdot (y\cdot z)$;
  \item[$\mathbf{Q}_{11}$] $x + y = y + x$;
  \item[$\mathbf{Q}_{12}$] $x\cdot y = y \cdot x$;
  \item[$\mathbf{Q}_{13}$] $x\leq y\leftrightarrow \exists z (x + z = y)$.
\end{description}
The theory $\mathbf{Q}^+$ is interpretable in $\mathbf{Q}$ (see \cite[Theorem 1]{Interpretability in Robinson's Q}, p.296), and hence $\mathbf{Q}^+$ is mutually interpretable with $\mathbf{Q}$.

Andrzej Grzegorczyk considered a theory $\mathbf{Q}^{-}$ in which addition and multiplication
do satisfy natural reformulations of axioms of $\mathbf{Q}$ but are possibly non-total
functions. More exactly, the language of $\mathbf{Q}^{-}$ is $\{\mathbf{0}, \mathbf{S}, A, M\}$ where $A$ and $M$
are ternary relations, and the axioms of $\mathbf{Q}^{-}$ are the axioms $\mathbf{Q}_1$-$\mathbf{Q}_3$ of $\mathbf{Q}$ plus
the following six axioms about $A$ and $M$:
\begin{description}
  \item[A] $\forall x\forall y\forall z_1\forall z_2(A(x, y, z_1)\wedge A(x, y, z_2) \rightarrow z_1 = z_2)$;
  \item[M] $\forall x\forall y\forall z_1\forall z_2(M(x, y, z_1) \wedge M(x, y, z_2) \rightarrow z_1 = z_2)$;
  \item[G4] $\forall x \, A(x, \mathbf{0}, x)$;
  \item[G5] $\forall x\forall y\forall z(\exists u(A(x, y, u) \wedge z = \mathbf{S}(u)) \rightarrow A(x, \mathbf{S}(y), z))$;
  \item[G6] $\forall x\, M(x, \mathbf{0}, \mathbf{0})$;
  \item[G7] $\forall x\forall y\forall z(\exists u(M(x, y, u) \wedge A(u, x, z)) \rightarrow M(x, \mathbf{S}(y), z))$.
\end{description}

Andrzej Grzegorczyk asked whether $\mathbf{Q}^{-}$ is essentially undecidable. Petr H\'{a}jek
considered a somewhat stronger theory with axioms
\begin{description}
  \item[H5] $\forall x\forall y\forall z(\exists u(A(x, y, u) \wedge z = \mathbf{S}(u)) \leftrightarrow A(x, \mathbf{S}(y), z))$ and
  \item[H7] $\forall x\forall y\forall z(\exists u(M(x, y, u) \wedge A(u, x, z)) \leftrightarrow M(x, \mathbf{S}(y), z))$
\end{description}
instead of $\mathbf{G5}$ and $\mathbf{G7}$. He showed that this stronger variant of $\mathbf{Q}^{-}$ is essentially undecidable (see \cite{Hajek 07}).
V\'{i}t\v{e}zslav \v{S}vejdar provided a positive answer to Grzegorczyk's original question in \cite{Svejdar 07} and proved that $\mathbf{Q}$ is interpretable in $\mathbf{Q}^{-}$ using the Solovay's method of shortening  cuts (and hence $\mathbf{Q}^{-}$ is essentially undecidable). Thus, $\mathbf{Q}^{-}$ is mutually interpretable with $\mathbf{Q}$.

Andrzej Grzegorczyk proposed the theory of concatenation ($\mathbf{TC}$) in \cite{Grzegorczyk 05} as a possible alternative theory for studying incompleteness
and undecidability. Unlike Robinson (or Peano) Arithmetic, where
the individuals are numbers that can be added or multiplied, in $\mathbf{TC}$ one has strings (or texts) that can be concatenated. We refer to \cite{Grzegorczyk 05} for Grzegorczyk's philosophical motivations to study $\mathbf{TC}$.

%The motivations for accepting strings rather than numbers as the
%basic notion can briefly be summarized as follows: in G\"{o}del's argument, the only use of numbers is coding of syntactical objects; dealing with texts is philosophically better justified because intellectual activities like reasoning, communicating, or even computing involve working with texts, not with numbers;  on a metamathematical level, the notion of computability can be
%defined without reference to numbers. Thus it is natural to
%define notions like undecidability directly in terms of texts, without
%reference to natural numbers.

The theory $\mathbf{TC}$ has the language $\{\frown, \alpha, \beta, =\}$ with a binary function
symbol and two constants, and the following axioms:
\begin{description}
  \item[TC1] $\forall x\forall y\forall z(x\frown (y\frown z) = (x\frown y)\frown z)$;
  \item[TC2] $\forall x\forall y\forall u\forall v(x\frown y = u\frown v \rightarrow ((x = u \wedge y = v) \vee\exists w((u = x\frown w \wedge w\frown v= y) \vee (x = u\frown w \wedge w\frown y = v))))$;
  \item[TC3] $\forall x\forall y(\alpha \neq x\frown y)$;
  \item[TC4] $\forall x\forall y(\beta \neq x\frown y)$;
  \item[TC5] $\alpha \neq\beta$.
\end{description}

Grzegorczyk \cite{Grzegorczyk 05} proved (mere) undecidability of the theory $\mathbf{TC}$. Grzegorczyk and Zdanowski \cite{Grzegorczyk 08} proved that $\mathbf{TC}$ is essentially undecidable. However, \cite{Grzegorczyk 08} leaves an interesting unanswered question: are $\mathbf{TC}$ and $\mathbf{Q}$ mutually interpretable? \v{S}vejdar \cite{Svejdar 09} showed  that $\mathbf{Q}^{-}$ is interpretable in  $\mathbf{TC}$ and hence $\mathbf{Q}$ is interpretable in $\mathbf{TC}$ since $\mathbf{Q}$ is interpretable in $\mathbf{Q}^{-}$.
Ganea  \cite{Ganea 09} gave a different proof   of the interpretability of $\mathbf{Q}$ in $\mathbf{TC}$,
but he also used the detour via $\mathbf{Q}^{-}$. Visser \cite{Visser 09} gave a proof of the interpretability of $\mathbf{Q}$ in $\mathbf{TC}$ not using $\mathbf{Q}^{-}$. Note that $\mathbf{TC}$ is easily interpretable in the bounded arithmetic $I\Sigma_0$. Thus, $\mathbf{TC}$ is mutually interpretable with $\mathbf{Q}$.

Adjunctive Set Theory ($\mathbf{AS}$) is the following theory in the language with only one binary relation
symbol $\in$.
\begin{description}
  \item[AS1] $\exists x \forall y (y \notin x)$.
  \item[AS2] $\forall x\forall y \exists z \forall u (u \in z \leftrightarrow (u = x \vee u = y))$.
\end{description}

The theory $\mathbf{AS}$ interprets Robinson's Arithmetic $\mathbf{Q}$, and hence is essentially undecidable.
Nelson \cite{Nelson 86} showed that $\mathbf{AS}$ is interpretable in  $\mathbf{Q}$. Thus, $\mathbf{AS}$ is mutually interpretable with $\mathbf{Q}$.

The theory $\mathbf{S^1_2}$ is a finitely axiomatizable weak arithmetic introduced by Buss in \cite{Buss 86} to study polynomial time computability. The theory $\mathbf{S^1_2}$ gives us what we need to formalize the proof of the Second Incompleteness Theorem in a
natural and effortless way. In fact, it is easier to do it in $\mathbf{S^1_2}$ than in
$\mathbf{PA}$, since the restrictions present in $\mathbf{S^1_2}$ prevent one from making wrong turns and inefficient choices (see \cite{paper on Q}). From \cite{Interpretability in Robinson's Q},  $I\Sigma_0$ is interpretable in  $\mathbf{S^1_2}$  and $\mathbf{S^1_2}$ is interpretable in $\mathbf{Q}$. Thus, $\mathbf{S^1_2}$ is mutually interpretable with $\mathbf{Q}$.

%In a summary,  theories $\mathbf{TC}, \mathbf{Q}^{-}, \mathbf{Q}^{+}, \mathbf{PA}^{-}, \mathbf{AS}, \mathbf{S^1_2}$ and %$\mathbf{Q}$  are all mutually interpretable and essentially
%undecidable.
%, and hence $\sf G1$ holds for all these theories.

Now, we introduce Robinson's theory $\mathbf{R}$ introduced by Tarski, Mostowski and Robinson in \cite{undecidable}, and some variants of $\mathbf{R}$.
\begin{definition}
Let $\mathbf{R}$ be the theory consisting of schemes $\sf{Ax1}$-$\sf{Ax5}$ with $L(\mathbf{R})=\{\mathbf{0}, \mathbf{S}, +, \cdot, \leq\}$ where  $\leq$ is a primitive binary relation  symbol and  $\overline{n}=\mathbf{S}^n \mathbf{0}$ for $n \in \mathbb{N}$:
\begin{description}
  \item[\sf{Ax1}] $\overline{m}+\overline{n}=\overline{m+n}$;
  \item[\sf{Ax2}] $\overline{m}\cdot\overline{n}=\overline{m\cdot n}$;
  \item[\sf{Ax3}] $\overline{m}\neq\overline{n}$, if $m\neq n$;
  \item[\sf{Ax4}] $\forall x(x\leq \overline{n}\rightarrow x=\overline{0}\vee \cdots \vee x=\overline{n})$;
  \item[\sf{Ax5}] $\forall x(x\leq \overline{n}\vee \overline{n}\leq x)$.
\end{description}
\end{definition}

%\section{G\"{o}del's first incompleteness theorem in $\mathbf{R}^{\ast}$}

%Let $\mathbf{R}$ be the system consisting of schemes $Ax1$-$Ax5$.
The axiom schemes of $\mathbf{R}$ contain all key properties of arithmetic for the proof of $\sf G1$.
%Note that $\mathbf{R}\subseteq \mathbf{Q}\subseteq\mathbf{PA}$.
The theory $\mathbf{R}$ is not finitely axiomatizable. Note that $\mathbf{R}\lhd\mathbf{Q}$ since $\mathbf{Q}$ is not interpretable in $\mathbf{R}$: if $\mathbf{Q}$ is interpretable in $\mathbf{R}$, then it is interpretable in some finite fragment of $\mathbf{R}$; however  $\mathbf{R}$ is locally finitely satisfiable but any model of $\mathbf{Q}$ is infinite. Visser \cite{Visser 14} proved  the following universal property of $\mathbf{R}$ which provides a unique characterization of $\mathbf{R}$.

\begin{theorem}[Visser, Theorem 6, \cite{Visser 14}]\label{visser thm on R}
For any r.e.~ theory $T$,  $T$ is locally finitely satisfiable iff $T$ is interpretable in $\mathbf{R}$.\footnote{In fact, if $T$ is locally finitely satisfiable, then $T$ is interpretable in $\mathbf{R}$ via a one-piece one-dimensional parameter-free interpretation.}
\end{theorem}

We say a specific class $\Phi$ of sentences  has the finite
model property if every satisfiable sentence of $\Phi$ has a finite
model. Since relational $\Sigma_2$ sentences in a finite relational language have the finite model property (see Chapter 5 in \cite{Finite Model Theory 99}), by Theorem \ref{visser thm on R},  any consistent theory axiomatized by a recursive set of $\Sigma_2$ sentences in a finite relational language is interpretable in $\mathbf{R}$. Since all recursive functions are representable in $\mathbf{R}$ (see \cite[Theorem 6]{undecidable}, p.56), from Theorem \ref{interpretable theorem}(2),  $\mathbf{R}$ is  essentially undecidable. Cobham showed that $\mathbf{R}$ has a stronger property than essentially undecidability.
\begin{theorem}[Cobham, \cite{Vaught 62}]\label{Cobham theorem}
Any r.e.~ theory that weakly interprets $\mathbf{R}$ is undecidable.\footnote{Vaught \cite{Vaught 62} gave a proof of Cobham's theorem  via existential interpretation.}
\end{theorem}

%However, $\mathbf{R}$ is not minimal essentially undecidable.
Now, we discuss some variants of $\mathbf{R}$. If not
explicitly mentioned otherwise, we assume that the base language is the same as $L(\mathbf{R})=\{\mathbf{0}, \mathbf{S}, +, \cdot, \leq\}$ with $\leq$ as a primitive binary relation symbol.
Let $\mathbf{R}_0$ be the theory consisting of schemes $\sf{Ax1}, \sf{Ax2}, \sf{Ax3}$ and $\sf{Ax4}$.
The theory $\mathbf{R}_0$ is no longer essentially undecidable: the theory $\mathbf{R}_0$ has a decidable complete extension given by the theory of real closed fields with $\leq$ as the empty relation on reals. In fact, whether $\mathbf{R}_0$ is  essentially undecidable depends on the language of $\mathbf{R}_0$. If $L(\mathbf{R}_0)=\{\mathbf{0}, \mathbf{S}, + , \cdot\}$  with $\leq$  defined in terms
of $+$, then $\mathbf{R}_0$ is essentially undecidable: Cobham first observed that  $\mathbf{R}$  is interpretable in $\mathbf{R}_0$ and hence $\mathbf{R}_0$ is essentially undecidable (see \cite{Vaught 62} and \cite{Jones 83}).
%\footnote{If we delete any axiom scheme of $\mathbf{R}_0$, then the remaining system is not essentially undecidable: if we delete $Ax1$, then the theory of the reals with the non-standard definition of $x+y = -1$ for all x and y is a complete decidable extension;  if we delete $Ax2$, then the theory of natural numbers with $x\cdot y$ defined as $x+y$ is a complete  decidable extension; if we delete $Ax3$, then the theory of models with only one element is a complete  decidable extension; if we delete $Ax4$, then the theory of reals is a complete  decidable extension. See \cite{Jones 83}, p.62.}
Let $\mathbf{R}_1$ be the theory consisting of schemes $\sf{Ax1}, \sf{Ax2}, \sf{Ax3}$ and $\sf{Ax4}^{\prime}$ where $\sf{Ax4}^{\prime}$ is defined as follows:
\[\forall x(x\leq \overline{n}\leftrightarrow x=\overline{0}\vee \cdots \vee x=\overline{n}).\]
The theory $\mathbf{R}_1$ is  essentially undecidable since $\mathbf{R}$ is interpretable in $\mathbf{R}_1$ (see \cite{Jones 83}, p. 62).
However $\mathbf{R}_1$ is not minimal essentially undecidable. Let $\mathbf{R}_2$ be the system consisting of schemes $\sf{Ax2}, \sf{Ax3}$ and $\sf{Ax4}^{\prime}$.
From \cite{Jones 83}, $\mathbf{R}$ is interpretable in $\mathbf{R}_2$ and hence $\mathbf{R}_2$ is essentially undecidable.\footnote{Another way to show that $\mathbf{R}_2$ is essentially
undecidable is to prove that all recursive functions are representable in $\mathbf{R}_2$.} The theory $\mathbf{R}_2$  is minimal essentially undecidable in the sense that if we delete any axiom scheme of $\mathbf{R}_2$, then the remaining theory  is not essentially undecidable: if we delete $\sf{Ax2}$, then the theory of natural numbers with $x\cdot y$ defined as $x+y$ is a complete  decidable extension; if we delete $\sf{Ax3}$, then the theory of models with only one element is a complete  decidable extension; if we delete $\sf{Ax4}^{\prime}$, then the theory of real closed fields is a complete  decidable extension.
By essentially the same argument as the proof of Theorem \ref{Cobham theorem} in \cite{paper on Q}, we can show that any r.e.~ theory that weakly interprets $\mathbf{R}_2$ is undecidable.

Kojiro Higuchi and Yoshihiro Horihata introduced the theory of concatenation $\mathbf{WTC}^{-\epsilon}$, which is a weak subtheory
of Grzegorczyk's theory $\mathbf{TC}$, and showed that $\mathbf{WTC}^{-\epsilon}$ is  minimal essentially undecidable and $\mathbf{WTC}^{-\epsilon}$ is mutually interpretable with $\mathbf{R}$ (see \cite{Higuchi}).

Elementary Arithmetic ($\mathbf{EA}$) is  $I\Delta_0 +\mathbf{exp}$.
We refer to \cite{Metamathematics of First-Order
Arithmetic} for the definition of Primitive Recursive Arithmetic ($\mathbf{PRA}$). In a summary,  we have the following pictures.
\begin{enumerate}[(1)]
  \item Theories $\mathbf{PA}^{-}, \mathbf{Q}^{+}, \mathbf{Q}, \mathbf{Q}^{-},  \mathbf{TC}, \mathbf{AS}$ and $\mathbf{S^1_2}$  are  mutually interpretable   and are all essentially
undecidable.
  \item Theories $\mathbf{R}, \mathbf{R}_1, \mathbf{R}_2$ and $\mathbf{WTC}^{-\epsilon}$ are mutually interpretable and are all essentially undecidable.
\item  $\mathbf{R}\lhd\mathbf{Q}\lhd \mathbf{EA}\lhd \mathbf{PRA}\lhd \mathbf{PA}$.
\end{enumerate}

Now, a natural question is: among finitely axiomatized theories for which $\sf G1$ holds, does $\mathbf{Q}$ have the least degree of interpretation? The following theorem tells us that the answer is no.

\begin{theorem}[Visser, Theorem 2, \cite{paper on Q}]\label{Visser on Q}
Suppose $\mathbf{R} \subseteq A$, where $A$ is finitely axiomatized
and consistent. Then, there is a finitely axiomatized $B$ such
that $\mathbf{R} \subseteq B \subseteq A$ and $B\lhd A$.
\end{theorem}
Theorem \ref{Visser on Q}  shows that the structure $\langle \{S: \mathbf{R}\unlhd S\lhd \mathbf{Q}\}, \lhd\rangle$ is not well founded w.r.t.~ finitely axiomatized theories.

\begin{theorem}[Visser, Theorem 12, \cite{paper on Q}]\label{incomparable thm}
Suppose $A$ and $B$ are finitely axiomatized theories that interpret   $\mathbf{S^1_2}$. Then there are finitely axiomatized theories $\overline{A}\supseteq A$ and $\overline{B}\supseteq B$ such that $\overline{A}$ and $\overline{B}$ are incomparable (i.e.~ $\overline{A}\ntrianglelefteq \overline{B}$ and $\overline{B}\ntrianglelefteq \overline{A}$).
\end{theorem}
Theorem \ref{incomparable thm} shows that there are many pairs of incomparable theories  extending $\mathbf{Q}$.

\section{Finding the limit of applicability of $\sf G1$ w.r.t.~ interpretation}\label{section 3}

%As the first step toward Question \ref{key qn}, a natural question is:  could we find a theory  $S$ such that $\sf G1$  holds for $S$  and $S\lhd \mathbf{R}$?
In this section, we examine the  limit of applicability of $\sf G1$ w.r.t.~ interpretation and show that we can find many  theories weaker than $\mathbf{R}$ w.r.t.~ interpretation for which $\sf G1$ holds  based on Je\v{r}\'{a}bek's work using some model theory.  First of all, we give some equivalent characterizations of the notion ``$\sf G1$ holds for the theory $T$".
%explains why I examine essentially undecidable theories weaker than $\mathbf{PA}$ w.r.t. interpretation in the literature in Section \ref{section 2}.
\begin{proposition}\label{eq pro}
Let $T$ be a recursively axiomatizable consistent theory. The following are equivalent:
\begin{enumerate}[(1)]
  \item  $\sf G1$ holds for  $T$.
  \item $T$ is essentially incomplete.
  \item $T$ is essentially undecidable.
\end{enumerate}
\end{proposition}
\begin{proof}\label{}
$(1) \Rightarrow (2)$ is  trivial.

$(2) \Leftrightarrow (3)$: It is well known that every consistent recursively axiomatizable
complete theory is decidable; and every incomplete decidable theory has a consistent,
decidable complete extension in the same language (see Corollary 3.1.8 and Theorem 3.1.9 in \cite{metamathematics}, p.214-215). From these two facts, $T$ is essentially undecidable  iff $T$ is essentially incomplete.

 $(2) \Rightarrow (1)$: Follows from Theorem \ref{interpretable theorem} and $(2) \Leftrightarrow (3)$.
\end{proof}

As a corollary of Section \ref{section 2}, we have:
\begin{enumerate}[(1)]
  \item $\sf G1$  holds for the following theories and they are mutually interpretable:
$\mathbf{Q}, I\Sigma_0, I\Sigma_0+\Omega_{n}(n\geq 1), B\Sigma_1, B\Sigma_1+\Omega_{n}(n\geq 1), \mathbf{TC}, \mathbf{Q}^{-}, \mathbf{Q}^{+}, \mathbf{PA}^{-}, \mathbf{S^1_2}$.
  \item $\sf G1$  holds for the following theories and they are mutually interpretable: $\mathbf{R}, \mathbf{R}_1, \mathbf{R}_2$ and $\mathbf{WTC}^{-\epsilon}$.
      %\item $\sf G1$  holds for $\mathbf{Rep_{\sf PRF}}, \mathbf{R}, \mathbf{Q}, \mathbf{EA}, \mathbf{PRA}, \mathbf{PA}$ and $\mathbf{Rep_{\sf PRF}}\lhd \mathbf{R}$.
\end{enumerate}

Up to now, we do not have any
example of an essentially undecidable theory $S$  such that $S \lhd \mathbf{R}$ and $\sf G1$  holds for $S$. %$S$ such that $S\lhd \mathbf{R}$.
We find that Je\v{r}\'{a}bek's work in \cite{Emil} essentially provides such an example of theory $S$.
The motivating question of \cite{Emil} is: if a theory represents all partial recursive functions, does it interpret Robinson's theory $\mathbf{R}$?  Je\v{r}\'{a}bek \cite{Emil} negatively answered this question and
showed    that there exists a theory $T$ in which all partial
recursive functions are representable, yet $T$ does not interpret $\mathbf{R}$.
 Je\v{r}\'{a}bek's proof uses tools from model theory: investigating model-theoretic
properties of the model completion of the empty theory in a language with function symbols (see \cite{Emil}).\footnote{By the empty theory, we mean the theory with no extra-logical axioms.}
Now we introduce Je\v{r}\'{a}bek's theory $\mathbf{Rep_{\sf PRF}}$.
Let $\sf PRF$ denote the set of all partial recursive functions.  The language $L(\mathbf{Rep_{\sf PRF}})$ consists of constant symbols $\bar{n}$ for each $n \in \mathbb{N}$ and function symbols $\overline{f}$ of appropriate arity for each partial recursive function $f$. % and  relation symbols $\overline{P}$ for every disjoint pair of r.e. relations $\langle P_0, P_1\rangle \in  A$.
The theory $\mathbf{Rep_{\sf PRF}}$ is axiomatized by:
\begin{enumerate}[(1)]
  \item $\overline{n} \neq \overline{m}$ for $n \neq m \in \mathbb{N}$;
  \item $\overline{f}(\overline{n_0}, \cdots, \overline{n_{k-1}}) = \overline{m}$
for each $k$-ary partial recursive function $f$ such that $f(n_0, \cdots, n_{k-1}) = m$ where $n_0, \cdots, n_{k-1}, m \in \mathbb{N}$.
  %\item  For each $k$-ary
%disjoint pair of r.e. relations $\langle P_0, P_1\rangle \in  A$, we have:
%\begin{enumerate}[(a)]
  %\item $\overline{P}(\overline{n_0}, \cdots, \overline{n_{k-1}})$
%for $\langle n_0, \cdots, n_{k-1}\rangle \in P_0$, and
  %\item
%$\neg\overline{P}(\overline{n_0}, \cdots, \overline{n_{k-1}})$ for $\langle n_0, \cdots, n_{k-1}\rangle \in P_1$.
%\end{enumerate}
\end{enumerate}

\begin{corollary}\label{answer of key qn}
$\sf G1$ holds for $\mathbf{Rep_{\sf PRF}}$ and $\mathbf{Rep_{\sf PRF}}\lhd \mathbf{R}$.
\end{corollary}
\begin{proof}\label{}
The theory $\mathbf{Rep_{\sf PRF}}$ is essentially undecidable since all recursive functions are representable in it. Since $\mathbf{Rep_{\sf PRF}}$ is locally finitely satisfiable, by Theorem \ref{visser thm on R}, $\mathbf{Rep_{\sf PRF}}\unlhd \mathbf{R}$. Je\v{r}\'{a}bek \cite{Emil} showed that $\mathbf{R}$ is not interpretable in
$\mathbf{Rep_{\sf PRF}}$. Thus, $\sf G1$ holds for $\mathbf{Rep_{\sf PRF}}$ and $\mathbf{Rep_{\sf PRF}}\lhd \mathbf{R}$.
\end{proof}

Je\v{r}\'{a}bek's \cite{Emil} is not written in the spirit of answering Question \ref{main qn} and the potential of the method in \cite{Emil} is not yet fully explored. Now, we give more examples of a theory $S$ such that $\sf G1$ holds for $S$ and $S\lhd \mathbf{R}$ based on Je\v{r}\'{a}bek's work in \cite{Emil}.

\begin{definition}
We say $\langle S, T\rangle$ is a recursively inseparable pair if $S$ and $T$ are disjoint r.e.~ subsets of $\mathbb{N}$, and there is no recursive set $X\subseteq\mathbb{N}$ such that $S\subseteq X$ and $X\cap T=\emptyset$.
\end{definition}

\begin{definition}
Let $\langle A, B\rangle$ be a recursively inseparable pair. Consider the following r.e.~ theory $U_{\langle A,B\rangle}$ with $L(U_{\langle A,B\rangle})=\{\mathbf{0},  \mathbf{S}, \mathbf{P}\}$ where $\mathbf{P}$ is a unary relation symbol and $\overline{n}=\mathbf{S}^n\mathbf{0}$ for $n\in \mathbb{N}$:
\begin{enumerate}[(1)]
  \item $\overline{m}\neq\overline{n}$ if $m\neq n$;
  \item $\mathbf{P}(\overline{n})$ if $n\in A$;
  \item $\neg \mathbf{P}(\overline{n})$ if $n\in B$.
\end{enumerate}
\end{definition}

In the following, let $\langle A, B\rangle$ be an arbitrary recursively inseparable pair.
\begin{lemma}\label{G1 holds}
$\sf G1$  holds for $U_{\langle A,B\rangle}$.
\end{lemma}
\begin{proof}\label{}
By Proposition \ref{eq pro}, it suffices to show that $U_{\langle A,B\rangle}$ is essentially incomplete.
let $S$ be a recursively axiomatizable consistent extension of $U_{\langle A,B\rangle}$. Let $X = \{n : S\vdash \mathbf{P}(\overline{n})\}$ and $Y = \{n : S\vdash\neg \mathbf{P}(\overline{n})\}$.  Then  $A\subseteq X$ and $B \subseteq Y$.
Since $S$ is recursively axiomatizable and consistent,  $X$ and $Y$ are disjoint recursive enumerable sets. Since $\langle A,B\rangle$ is recursively inseparable, $X\cup Y\neq \mathbb{N}$. Take $n \notin X\cup Y$. Then $S \nvdash \mathbf{P}(\overline{n})$ and $S\nvdash\neg \mathbf{P}(\overline{n})$.  Hence $S$ is incomplete.
\end{proof}

\begin{fact}[Theorem 2, p.43, \cite{Per 97}]\label{disjoint pair}
Let $A$ and $B$ be disjoint r.e.~ subsets of $\mathbb{N}$ and $T$ be a consistent r.e.~ extension of $\mathbf{Q}$. Then there is  a $\Sigma_1$ formula $\phi(x)$  such that for any $n$, we have:
\begin{enumerate}
\item  $n\in A$ iff $T\vdash\phi(\overline{n})$;
\item  $n\in B$ iff $T\vdash\neg\phi(\overline{n})$.
\end{enumerate}
\end{fact}

\begin{lemma}\label{interpretation}
The theory $U_{\langle A,B\rangle}$ is interpretable in $\mathbf{R}$.
\end{lemma}
\begin{proof}\label{}
By Fact \ref{disjoint pair}, there exists a formula $\phi(x)$ with one free variable such that $\mathbf{R}\vdash \phi(\overline{n})$ iff $n \in A$ and
$\mathbf{R}\vdash \neg\phi(\overline{n})$ iff $n \in B$.\footnote{The proof of Fact \ref{disjoint pair} in   \cite[Theorem 2, p.43]{Per 97} uses the fixed point theorem for the base theory $T$. Since the fixed point theorem holds for $\mathbf{R}$, Fact \ref{disjoint pair} also applies to $\mathbf{R}$.} Thus $U_{\langle A,B\rangle}$ is interpretable in $\mathbf{R}$ via interpreting $\mathbf{P}(x)$ as $\phi(x)$.

Here is another proof of Lemma \ref{interpretation}: since $U_{\langle A,B\rangle}$ is a locally finitely satisfiable r.e.~ theory, by Theorem \ref{visser thm on R}, $U_{\langle A,B\rangle}$ is interpretable in $\mathbf{R}$.
\end{proof}

%The proof of the main Theorem \ref{main thm} uses some tools from model theory as presented in \cite{Emil}.

Based on Je\v{r}\'{a}bek's work in \cite{Emil}, our strategy to prove that $U_{\langle A,B\rangle}$ does not  interpret $\mathbf{R}$ is to consistently extend the interpreting theory to a theory
with quantifier elimination, using the fact that the empty theory in an arbitrary language $L$ has a
model completion which we denote by $\mathbf{EC}_L$, the theory of existentially closed $L$-structures. The theory $\mathbf{EC}_{L}$  admits the elimination of quantifiers (see \cite{Emil}).
The two key tools we use to show that $U_{\langle A,B\rangle}$ does not  interpret $\mathbf{R}$ are  Theorem \ref{main tool} and Theorem \ref{general result} which essentially use properties of $\mathbf{EC}_L$.
We say a relation $R \subseteq X^2$ is \emph{asymmetric} if there are no $a, b \in X$ such that $R(a, b)$ and $R(b, a)$.

\begin{theorem}[Je\v{r}\'{a}bek, Theorem 5.1, \cite{Emil}]\label{main tool}~
For any first order language $L$ and formula $\phi(\overline{z}, \overline{x}, \overline{y})$ with $lh(\overline{x}) = lh(\overline{y})$, there is a constant $n$
with the following property. Let $M\models \mathbf{EC}_{L}$ and $\overline{u} \in M$ be such that
$M \models \exists \overline{x}_0, \cdots, \exists \overline{x}_{n-1}\bigwedge_{i<j<n}\phi(\overline{u}, \overline{x}_i, \overline{x}_j)$.
Then for every $m \in\mathbb{N}$ and an asymmetric relation $R$ on $\{0, \cdots, m-1\}$,
$M \models \exists \overline{x}_0, \cdots, \exists \overline{x}_{m-1}\bigwedge_{\langle i, j\rangle\in R}\phi(\overline{u}, \overline{x}_{i}, \overline{x}_{j})$.
\end{theorem}

\begin{theorem}[Je\v{r}\'{a}bek, Theorem 4.5, \cite{Emil}]\label{general result}
For a $\Sigma_2$-axiomatized theory $T$, $T$ is interpretable in a consistent existential theory iff $T$ is weakly interpretable in $\mathbf{EC}_L$ for some language $L$.
\end{theorem}
Especially, Je\v{r}\'{a}bek \cite{Emil} showed that: (1) if a theory is interpretable in a consistent quantifier-free or existential theory,
it is weakly interpretable in $\mathbf{EC}_{L}$ for some  language $L$, and the interpretation can be taken quantifier-free; (2) if a $\Sigma_2$ theory is weakly interpretable in $\mathbf{EC}_{L}$, it is interpretable in a quantifier-free theory.

\begin{definition}\label{theory T}
Consider the following theory $\mathbf{T}$ in the language $\langle\in\rangle$ axiomatized by the sentences
$\exists z, x_0, \cdots, x_{n-1}(\bigwedge_{i<j<n} x_i \neq x_j\wedge \forall y (y \in z \leftrightarrow\bigvee_{i<n}
y = x_i))$ for all $n \in\mathbb{N}$.
\end{definition}

\begin{proposition}\label{main corollary}~
\begin{enumerate}[(1)]
  \item $\mathbf{T}$  is not weakly interpretable in $\mathbf{EC}_{L}$ for any language $L$.
  \item $\mathbf{R}$ is  not weakly interpretable in $\mathbf{EC}_{L}$ for any language $L$.
  \item If $\mathbf{R}$ is  interpretable in $U_{\langle A,B\rangle}$, then $\mathbf{R}$ is  weakly interpretable in $\mathbf{EC}_{L}$ for some language $L$.
\end{enumerate}
\end{proposition}
\begin{proof}\label{}
(1): Suppose this does not hold and apply Theorem \ref{main tool} to the formula which interprets   $\bigwedge_{i<j<n} x_i \neq x_j\wedge \forall y (y \in z \leftrightarrow\bigvee_{i<n}
y = x_i)$ and  $R$ a chain longer than $n$ to get a contradiction.

(2):  Note that $\mathbf{T}$ is interpretable in $\mathbf{R}$. Since $\mathbf{T}$  is not weakly interpretable in $\mathbf{EC}_{L}$ for any language $L$, $\mathbf{R}$ is  not weakly interpretable in $\mathbf{EC}_{L}$ for any language $L$.

(3): This follows from Theorem \ref{general result} since $U_{\langle A,B\rangle}$ is a consistent  r.e.~ theory.
\end{proof}

\begin{theorem}\label{main thm}
For any recursively inseparable pair $\langle A,B\rangle$, there is a r.e.~ theory $U_{\langle A,B\rangle}$ such that $\sf G1$  holds for $U_{\langle A,B\rangle}$  and $U_{\langle A,B\rangle}\lhd\mathbf{R}$.\footnote{However, Theorem \ref{main thm} does not tell us more information about the theory  $U_{\langle S,T\rangle}$ and $U_{\langle U,V\rangle}$ for different recursively inseparable pairs $\langle S,T\rangle$ and $\langle U,V\rangle$: e.g.~ whether $U_{\langle S,T\rangle}$ and $U_{\langle U,V\rangle}$ have the same degree of interpretation or the same degree of Turing reducibility.}
\end{theorem}
\begin{proof}\label{}
By  Proposition \ref{main corollary}(2)-(3), $\mathbf{R}$ is not interpretable in $U_{\langle A, B\rangle}$. From Lemma \ref{G1 holds} and  Lemma \ref{interpretation}, we have $\sf G1$  holds for $U_{\langle A,B\rangle}$ and $U_{\langle A, B\rangle}\lhd\mathbf{R}$.
\end{proof}

\begin{corollary}
Let $S$ be a consistent existential theory. Then  the following are equivalent:
\begin{enumerate}[(1)]
  \item $\sf G1$  holds for $S$  and $S\lhd\mathbf{R}$ (i.e. $S$ is a solution for Question \ref{main qn});
  \item $S$ is  essentially undecidable and locally finitely satisfiable.
\end{enumerate}
\end{corollary}
\begin{proof}\label{}
As a corollary of Theorem \ref{general result} and Proposition \ref{main corollary}(2), any consistent existential theory does not interpret $\mathbf{R}$. Thus, from this, Proposition \ref{eq pro} and Theorem \ref{visser thm on R}, we have the equivalence.
\end{proof}

%From Corollary \ref{key conclusion}, we have  proved the main Theorem \ref{main thm}.

%\begin{proposition}
%If $\sf G1$  holds for both $S$ and $T$, then $\sf G1$  holds for $Inf(S,T)$ where $Inf$ is the infimum theory in the interpretability degrees.
%\end{proposition}

%Theorem \ref{main thm} provides many examples for Question \ref{main qn} based on Je\v{R}\'{a}bek's work in \cite{Emil} via model theory.
%In the following, I will give more examples  for Question \ref{main qn} and  improve Theorem \ref{main thm} based on Shoenfield's work via recursion theory.  The proof of  Theorem \ref{key theorem} is based on Shoenfield's  Theorem \ref{Shoenfield}. To make the reader have better sense of the essentially undecidable theory we have, I add the proof of Shoenfield's  Theorem \ref{Shoenfield} from \cite{Shoenfield 58} which essentially uses Theorem \ref{Shoenfield first}.

From Theorem \ref{visser thm on R}, $\mathbf{R}$ has the  universality property: every locally finitely satisfiable r.e.~ theory is interpretable in it. Albert Visser asked the following question:
\begin{question}[Visser]\label{}
Would $S$ with $S\unlhd \mathbf{R}$ such that $\sf G1$  holds for  $S$ share the universality property
of $\mathbf{R}$ that every locally finitely satisfiable r.e.~ theory is interpretable in it.
\end{question}
As a corollary of Theorem \ref{main thm}, the answer for this question is negative.
We have shown that for any recursively inseparable pair $\langle A,B\rangle$, there is a theory $U_{\langle A,B\rangle}$ such that $\sf G1$  holds for $U_{\langle A,B\rangle}$ and $U_{\langle A,B\rangle}\lhd \mathbf{R}$. The theory $\mathbf{R}$ is locally finitely satisfiable, but $\mathbf{R}$  is not interpretable in $U_{\langle A,B\rangle}$.
Take another example: the theory $\mathbf{T}$ as in Definition \ref{theory T} is locally finitely satisfiable, but $\mathbf{T}$ is not interpretable in $U_{\langle A,B\rangle}$ (if $\mathbf{T}$  is interpretable in $U_{\langle A,B\rangle}$, by Theorem \ref{general result},  $\mathbf{T}$  is  weakly interpretable in $\mathbf{EC}_{L}$ for some language $L$ which contradicts Proposition \ref{main corollary}(1)).
Thus, for any recursively inseparable pair $\langle A,B\rangle$, the theory $U_{\langle A,B\rangle}$ is a counterexample for Visser's Question. This shows  the speciality of $\mathbf{R}$: Theorem \ref{visser thm on R} provides a unique characterization of $\mathbf{R}$.

Define ${\sf D}=\{S: S\lhd \mathbf{R}$ and $\sf G1$  holds for the theory $S$\}. We have shown that we could find many witnesses for ${\sf D}$. We could naturally examine the  structure of $\langle {\sf D}, \lhd\rangle$.
A natural question is: whether the similar results as in Theorem \ref{Visser on Q} and Theorem \ref{incomparable thm} apply to the structure  $\langle {\sf D}, \lhd\rangle$. About the structure of $\langle {\sf D}, \lhd\rangle$, we could naturally ask:

\begin{question}\label{open qn}
\begin{enumerate}[(1)]~
  \item   Is $\langle {\sf D}, \lhd\rangle$ well founded (or is it that for any $S\in {\sf D}$, there is $T\in {\sf D}$ such that $T\lhd S$)?
  \item Are any two elements of $\langle {\sf D}, \lhd\rangle$ comparable (i.e.~ is it that for any $S,T\in {\sf D}$, we have either $S\unlhd T$ or $T\unlhd S$)?
  \item  Could we find a  theory $S$ with a minimal degree of  interpretation such that $\sf G1$  holds for $S$?
\end{enumerate}
\end{question}

In the rest of this paper, we will show that if we
consider the Turing degree structure  instead of the interpretation degree structure of ${\sf D}$, we have definite answers for Question \ref{open qn}.

\section{The limit of applicability of $\sf G1$ w.r.t.~ Turing reducibility}\label{section 3}

In this section, we examine the limit of applicability of $\sf G1$ w.r.t.~ Turing reducibility and show that there is no theory with a minimal degree of Turing reducibility for which $\sf G1$  holds based on Shoenfield's work using some recursion theory.

Let $\mathcal{R}$ be the structure of the r.e.~ degrees with the ordering $\leq_{T}$ induced by Turing
reducibility  with the least element $\mathbf{0}$ and the greatest element $\mathbf{0}^{\prime}$.
Define ${\sf \overline{D}}=\{S: S<_{T} \mathbf{R}$ and $\sf G1$  holds for the theory $S$\}. A natural question is to examine the structure of $\langle {\sf \overline{D}}, <_{T}\rangle$. Now, we will show that the structure $\langle {\sf \overline{D}}, <_{T}\rangle$ is much simpler than $\langle {\sf D}, \lhd\rangle$ and we have answers to Question \ref{open qn} for the structure $\langle {\sf \overline{D}}, <_{T}\rangle$ based on Shoenfield's work using some recursion theory.

\begin{theorem}[Shoenfield, Theorem 1, \cite{Shoenfield 58}]\label{Shoenfield first}
If $A$ is recursively enumerable and not recursive, there is a recursively inseparable pair $\langle B, C\rangle$ such that $A$, $B$ and $C$ have the same Turing degree.
\end{theorem}

Now, we will show that for any Turing degree $\mathbf{0}< \mathbf{d}<\mathbf{0}^{\prime}$, there is a theory $U$ such that $\sf G1$  holds for $U$, $U<_{T} \mathbf{R}$ and $U$ has Turing degree $\mathbf{d}$ (c.f. Theorem \ref{key theorem}). The following
theorem of Shoenfield is essential for the proof of Theorem \ref{key theorem}. To make the reader have a better sense of how the theory $U$ in Theorem \ref{key theorem} is constructed, we provide details of the proof of Theorem \ref{Shoenfield}. Feferman \cite{Feferman 57} also proved that for any r.e.~ Turing
degree one can design a formal theory whose corresponding decision problem
is of the same degree (however, it is not clear whether such a formal theory is essentially undecidable).

\begin{theorem}[Shoenfield, Theorem 2, \cite{Shoenfield 58}]\label{Shoenfield}
Let $A$ be recursively enumerable and not recursive. Then there is a consistent axiomatizable theory $T$ having one non-logical symbol which is essentially undecidable and has the same Turing degree as $A$.
\end{theorem}
\begin{proof}\label{}
By Theorem \ref{Shoenfield first}, pick a recursively inseparable pair $\langle B, C\rangle$ such that $A$, $B$ and $C$ have the same Turing degree. Now we define the theory $T$ with $L(T)=\{R\}$ where $R$ is a binary relation symbol. Theory $T$ contains axioms asserting
that $R$ is an equivalence relation. Let $\Phi_n$ be the statement that there is an
equivalence class of $R$ consisting of $n$ elements. Then, as axioms of $T$ we adopt $\Phi_n$ for all $n\in B$ and $\neg\Phi_n$ for all $n\in C$. Finally, for each $n$ we
adopt an axiom asserting there is at most one equivalence class of $R$ having $n$
elements.
Clearly, $T$ is consistent and axiomatizable. Using models, we see $\Phi_n$ is
provable iff $n\in B$, and $\neg\Phi_n$ is provable iff $n \in C$. Hence
$B$ and $C$ are recursive in $T$.

Disjunctions of conjunctions whose terms are $\Phi_n$ or $\neg \Phi_n$ for some $n\in\mathbb{N}$, are called a disjunctive normal form of $\langle\Phi_n: n\in\mathbb{N}\rangle$.
\begin{lemma}[Janiczak, Lemma 2 in \cite{Janiczak}]\label{Janiczak's Lemma}
Any sentence $\phi$ of the theory $T$ is equivalent to a disjunctive normal form  of $\langle\Phi_n: n\in\mathbb{N}\rangle$, and  this disjunctive normal form  can be found explicitly once  $\phi$ is explicitly given.\footnote{This is a reformulation of Janiczak's Lemma 2 in \cite{Janiczak} in the context of the theory $T$. Janiczak's Lemma is proved by means of a method known as the elimination of quantifiers.}
\end{lemma}

By Lemma \ref{Janiczak's Lemma},
every sentence $\phi$ of $T$ is equivalent to a disjunctive normal form of $\langle\Phi_n: n\in\mathbb{N}\rangle$, and
 this disjunctive normal form  can be calculated from $\phi$. It follows that $T$ is
recursive in $B$ and $C$. Hence $T$ has the same Turing degree as $A$.

Finally, we show that $T$ is essentially undecidable. Suppose $T$ has a consistent decidable
extension $S$. Let $D$ be the set of $n$ such that $\Phi_n$ is provable in $S$. Then $D$
is recursive, $B\subseteq D$, and $C\cap D = \emptyset$ which contradicts the fact that $\langle B, C\rangle$ is a recursively inseparable pair.
\end{proof}

%\begin{lemma}\label{degree lemma}
%If $\mathbf{R}$ is faithfully interpretable in a r.e.~ theory $S$, then $S$ has Turing degree $\mathbf{0}^{\prime}$.
%\end{lemma}
%\begin{proof}
%Let $I$ be the faithful interpretation of  $\mathbf{R}$ in  $S$. It suffices to show that any r.e.~ set $A$ is Turing reducible to $S$. By the $\Sigma_1$-completeness of $\mathbf{R}$, any r.e.~ set is representable in $\mathbf{R}$. Then there is a formula $\phi(x)$ such that for any $n\in\mathbb{N}$: (1) $\mathbf{R}\vdash\phi(\bar{n})$ if $n \in A$; (2) $\mathbf{R}\vdash\neg\phi(\bar{n})$ if $n \notin A$. Since $\mathbf{R}\vdash \phi(\bar{n})$ iff $S\vdash \phi^{I}(\bar{n})$ and $\mathbf{R}$ is consistent, we have $n\in A$ iff $S\vdash \phi^{I}(\bar{n})$. Thus, $A$ is Turing reducible to $S$.
%\end{proof}

\begin{theorem}[Sacks]\label{re degree thm}~
\begin{enumerate}[(1)]
  \item (Embedding theorem, \cite{Sacks 63})  Every countable partial ordering  can be embedded into $\mathcal{R}$.
  \item (Density Theorem, \cite{Sacks 64})  For every pair of nonrecursive r.e.~ degrees $\mathbf{a} <_{T} \mathbf{b}$, there is one $\mathbf{c}$ such that $\mathbf{a} <_{T} \mathbf{c} <_{T} \mathbf{b}$.
\end{enumerate}
\end{theorem}

\begin{theorem}\label{key theorem}
For any Turing degree $\mathbf{0}< \mathbf{d}<\mathbf{0}^{\prime}$, there is a theory $U$ such that $\sf G1$  holds for $U$, $U<_{T} \mathbf{R}$ and $U$ has Turing degree $\mathbf{d}$.
\end{theorem}
\begin{proof}
From Theorem \ref{Shoenfield}, for each Turing degree $\mathbf{0}< \mathbf{d}<\mathbf{0}^{\prime}$, there is a theory $U$ such that $\sf G1$  holds for $U$and $U$ has Turing degree $\mathbf{d}$. It is a well known fact that $\mathbf{R}$ has Turing degree $\mathbf{0}^{\prime}$.
\end{proof}

We could ask the similar question as Question \ref{open qn} for the structure $\langle {\sf \overline{D}}, <_{T}\rangle$:
\begin{itemize}
  \item Is $\langle {\sf \overline{D}}, <_{T}\rangle$ well founded?
  \item Are any two elements of $\langle {\sf \overline{D}}, <_{T}\rangle$ comparable?
  \item Could we find a theory $S$ with a minimal degree of Turing
reducibility such that ${\sf G1}$  holds for $S$?
\end{itemize}

From Theorem \ref{key theorem} and Theorem \ref{re degree thm}, we have answers for these questions:

\begin{corollary}\label{key corollary}
\begin{enumerate}[(1)]~
  \item The structure $\langle {\sf \overline{D}}, <_{T}\rangle$ is not well founded (i.e.~ for any $S\in {\sf \overline{D}}$, there is $U\in {\sf \overline{D}}$ such that $U<_{T} S$);
  \item  The structure $\langle {\sf \overline{D}}, <_{T}\rangle$ has incomparable elements (i.e.~ there are $U,V\in {\sf \overline{D}}$ such that $U\nleq_{T} V$ and $V\nleq_{T} U$);
  \item There is no  theory   with a minimal degree of Turing
reducibility for which $\sf G1$  holds.
\end{enumerate}
\end{corollary}

In fact, we can improve Theorem \ref{key theorem} by making that the theory $U$ is interpretable in $\mathbf{R}$.
\begin{theorem}\label{key theorem two}
For any Turing degree $\mathbf{0}< \mathbf{d}<\mathbf{0}^{\prime}$, there is a theory $U$ such that $\sf G1$  holds for $U$, $U\unlhd \mathbf{R}$ and $U$ has Turing degree $\mathbf{d}$.
\end{theorem}
\begin{proof}
Let $\mathbf{d}$ be a Turing degree with $\mathbf{0}< \mathbf{d}<\mathbf{0}^{\prime}$. By Theorem \ref{Shoenfield}, pick an essentially undecidable theory $S$ with Turing degree $\mathbf{d}$.

Consider the product theory $S\otimes \mathbf{R}$  defined as follows. The theory $S\otimes \mathbf{R}$ has the following axioms: $P\rightarrow X$ if $X$ is a $S$-axiom; $\neg P\rightarrow Y$ if $Y$ is a $\mathbf{R}$-axiom where $P$ is a 0-ary predicate symbol.

Now, we show that $S\otimes \mathbf{R}$ is essentially undecidable (i.e.~ $\sf G1$  holds for $S\otimes \mathbf{R}$) and interpretable in $\mathbf{R}$.

\begin{lemma}
$S\otimes \mathbf{R}$ is essentially undecidable.
\end{lemma}
\begin{proof}
Suppose $U$ is a consistent  decidable extension of $S\otimes \mathbf{R}$. Define $X=\{\langle\ulcorner\phi\urcorner, \ulcorner\psi\urcorner\rangle: U\vdash P\rightarrow \phi$ or $U\vdash \neg P\rightarrow \psi\}$. Since $U$ is decidable, $X$ is recursive. Note that $S\subseteq (X)_0$ and $\mathbf{R}\subseteq (X)_1$. We claim that at least one of $(X)_0$ and $(X)_1$ is consistent. If both $(X)_0$ and $(X)_1$ are inconsistent, then $U\vdash (P\rightarrow \perp)$ and $U\vdash (\neg P\rightarrow \perp)$. Thus, $U\vdash\perp$ which contradicts that $U$ is consistent. WLOG, we assume that $(X)_0$ is consistent. Then $(X)_0$ is consistent decidable extension of $S$ which contradicts that $S$ is essentially undecidable.
\end{proof}

It is easy to show that $S\otimes \mathbf{R}$ is interpretable in $\mathbf{R}$ (i.e. $S\otimes \mathbf{R}\unlhd\mathbf{R}$): take the identity interpretation on the $\mathbf{R}$ side and interpret $P$ as $\perp$.

Since $S$ has Turing degree $\mathbf{d}$ and $\mathbf{R}$ has Turing degree $\mathbf{0}^{\prime}$, $S\otimes \mathbf{R}$ has Turing degree $\mathbf{d}$.
\end{proof}

However, from the proof of Theorem \ref{key theorem two}, we cannot get that $S\otimes \mathbf{R}\lhd \mathbf{R}$ (i.e. $\mathbf{R}$ is not interpretable in $S\otimes \mathbf{R}$). An interesting question is: could we improve Theorem \ref{key theorem two} and show that for any Turing degree $\mathbf{0}< \mathbf{d}<\mathbf{0}^{\prime}$, there is a theory $U$ such that $\sf G1$  holds for $U$, $U\lhd \mathbf{R}$ and $U$ has Turing degree $\mathbf{d}$.

As far as we know, Question \ref{open qn}  is open. We make the conjecture that there is no theory  with a minimal  degree of interpretation for which  $\sf G1$  holds,  $\langle \sf D, \lhd\rangle$ is not well founded and $\langle \sf D, \lhd\rangle$ has incomparable elements.

%I conclude the paper with a conjecture.
%\begin{corollary}\label{key coro}
%$(\sf D, <_{T})$ is not well founded and has incomparable elements. %i.e. there exists an infinite decreasing chain in %$(\mathbf{D}, <)$.
%\end{corollary}

%\section{Questions}

%The answer is no: for
%any finitely axiomatized subtheory A of Q that extends R, we
%can find a finitely axiomatized subtheory B of A that extends
%R and such that B does not interpret A.

%Since $\mathbf{R}$ has Turing degree $0^{\prime}$, we have $S$ has Turing degree $0^{\prime}$.

%Define $\sf D=\{S: S\lhd \mathbf{R}$ and $\sf G1$  holds for $S\}$.

%\subsection{Incompleteness and provability logic}\label{provability logic}

%\subsection{Provability logic under numerations}

%First, we examine the characterization of truth and provability in provability logic.

%We consider the following conditions for Rosser provability predicates.

%\begin{definition}\label{}
%(\cite[Definition 3.1]{Kurahashi 2017}) For all formulas $\varphi$ and $\psi$:
%\begin{description}
  %\item[$D1^{R}$] If $T \vdash \varphi$, then $PA \vdash Pr_{T}^R(\ulcorner\varphi\urcorner)$.
  %\item[$D2^{R}$] $PA \vdash Pr_{T}^R(\ulcorner\varphi\rightarrow\psi\urcorner) \rightarrow (Pr_{T}^R(\ulcorner\varphi\urcorner) \rightarrow Pr_{T}^R(\ulcorner\psi\urcorner))$.
  %\item[$D3^{R}$] $PA \vdash Pr_{T}^R(\ulcorner\varphi\urcorner) \rightarrow \vdash Pr_{T}^R(\ulcorner Pr_{T}^R(\ulcorner\varphi\urcorner)\urcorner)$.
  %\item[
%\end{description}
%\end{definition}.

\end{document}